\documentclass[12pt,english]{amsart}
\usepackage{amsfonts,amsmath,amsthm,amscd,amssymb,latexsym,amsbsy,pb-diagram,cite,caption,url}
\usepackage[mathscr]{eucal}
\usepackage[T2A]{fontenc}
\usepackage[cp1251]{inputenc}
\usepackage[english]{babel}

\newtheorem{thm}{Theorem}[section]
\newtheorem{cor}[thm]{Corollary}
\newtheorem{lem}[thm]{Lemma}
\newtheorem{prop}[thm]{Proposition}
\newtheorem{conjecture}[thm]{Conjecture}
\theoremstyle{definition}
\newtheorem{defn}[thm]{Definition}
\theoremstyle{remark}
\newtheorem{rem}[thm]{Remark}
\numberwithin{equation}{section}
\newcommand{\B}{\textbf{B}}
\newcommand{\mfu}{\mathfrak{U}}
\DeclareMathOperator{\diam}{diam}

\newcommand{\abs}[1]{\left\vert#1\right\vert}

\newcommand{\norm}[1]{\operatorname{Sp}(#1)}
\newcommand{\Sp}[1]{\operatorname{Sp}(#1)}

\begin{document}

\title[On spaces extremal for the Gomory-Hu inequality]
 {On spaces extremal for the Gomory-Hu inequality}

\author{O. Dovgoshey}
\address{Institute of Applied Mathematics and Mechanics NASU, R. Luxemburg str. 74, Donetsk 83114, Ukraine
}
\address{Mersin University Faculty of Sciences and Arts, Department of Mathematics, Mersin 33342, Turkey
}
\email{aleksdov@mail.ru}

\author{E. Petrov}

\address{Institute of Applied Mathematics and Mechanics NASU, R. Luxemburg str. 74, Donetsk 83114, Ukraine}

\email{eugeniy.petrov@gmail.com}

\author{H.-M. Teichert}

\address{Institute of Mathematics, University of L\"ubeck, Ratzeburger Allee 160, 23562 L\"ubeck, Germany}

\email{teichert@math.uni-luebeck.de}

%\thanks{This work was completed with the support of an Izaak Walton Killam Memorial Scholarship.}

%\thanks{The author was also supported in part by the Research Council of Slovenia.}



%\date{November 18, 2014.}

%\dedicatory{}

%\commby{Daniel J. Rudolph}

% -----------------------------------------------------------

\begin{abstract}
Let $(X,d)$ be a finite ultrametric space. In 1961 E.C. Gomory and T.C. Hu proved
%for the finite ultrametric spaces $(X,d)$
the inequality $\abs{\Sp{X}}\leqslant\abs{X}$ where $\Sp{X}=\{d(x,y)\colon x,y \in  X\}$. Using weighted Hamiltonian cycles and weighted Hamiltonian paths we give new necessary and sufficient conditions under which the Gomory-Hu inequality becomes an equality.
We find the number of non-isometric $(X,d)$ satisfying the equality $|\norm{X}|=|X|$ for given $\norm{X}$. Moreover it is shown that every finite semimetric space $Z$ is an image under a composition of mappings $f\colon X\to Y$ and $g\colon Y\to Z$ such that $X$ and $Y$ are finite ultrametric space, $X$ satisfies the above equality, $f$ is an $\varepsilon$-isometry with an arbitrary $\varepsilon>0$, and $g$ is a ball-preserving map.

%We give a new criterium in terms of weighted cycles for finite ultrametric spaces belonging to the class of spaces for which equality is attained in the Gomory-Hu inequality.
%
\end{abstract}

% -----------------------------------------------------------
\maketitle
% -----------------------------------------------------------
\section{Introduction}
Recall some necessary definitions from the theory of metric spaces.
An \textit{ultrametric} on a set $X$ is a function $d\colon X\times X\rightarrow \mathbb{R}^+$, $\mathbb R^+ = [0,\infty)$, such that for all $x,y,z \in X$:
\begin{itemize}
\item [(i)] $d(x,y)=d(y,x)$,
\item [(ii)] $(d(x,y)=0)\Leftrightarrow (x=y)$,
\item [(iii)] $d(x,y)\leq \max \{d(x,z),d(z,y)\}$.
\end{itemize}
Inequality (iii)  is often called the {\it strong triangle inequality}.
By studying the flows in networks, R. Gomory and T. Hu \cite{GomoryHu(1961)}, deduced an inequality that can be formulated, in the language of ultrametric spaces, as follows: if $(X,d)$ is a finite nonempty ultrametric
space with the \emph{spectrum}
\begin{equation*}
\Sp{X}=\{d(x,y)\colon x,y \in  X\},
\end{equation*}
then
\begin{equation*}%\label{eq1}
\abs{\Sp{X}}\leqslant\abs{X}.
\end{equation*}
%In \cite{GomoryHu(1961)}, inequality~(\ref{eq1}) was obtained by studying the spanning trees of the least weight in complete weighted graphs.
\begin{defn}
Define by  $\mfu$ the class of finite ultrametric spaces $X$ with $\abs{\Sp{X}} = \abs{X}$.
\end{defn}

%The two criteria characterizing ultrametric spaces $X\in \mfu$ are proved in~\cite{PD(UMB)}. The first one is given by describing the structural properties of %some graphs related to spaces $X$ and the second one is given in terms of representing trees of finite ultrametric spaces (Theorems 2.3 and 3.1).

Two descriptions of $X \in \mfu$ were obtained in terms of the representing trees an, respectively, so-called diametrical graphs of $X$ (see~\cite{PD(UMB)} theorems 2.3 and 3.1.).
%In the section~\ref{cr}
Our paper is also a contribution to this lines of studies. We give a new criterium of $X \in \mfu$ in terms of weighted Hamiltonian cycles and weighted Hamiltonian paths
%related $X$. A connection between representing tree $T_X$
%of finite ultrametric space $X$
% and special type of cycles in $X$ is found in
(see Theorem~\ref{th1}) and find the number of non-isometric $X \in \mathfrak U$ with given $\Sp{X}$ (see Proposition \ref{p2}). It is also shown that every finite semimetric $X$ is an image of a space $Y \in \mathfrak U$, $X = g(f(Y))$, where $g$ is a ball-preserving map and $f$ is an $\varepsilon$-isometry (see Theorem \ref{th4.2} and Theorem \ref{th4.4}).

Recall that a \textit{graph} is a pair $(V,E)$ consisting of nonempty set $V$ and (probably empty) set $E$  elements of which are unordered pairs of different points from $V$. For the graph $G=(V,E)$, the set $V=V(G)$ and $E=E(G)$ are called \textit{the set of vertices} and \textit{the set of edges}, respectively. A graph $G$ is empty if $E(G) = \varnothing$. A graph is complete if $\{x,y\} \in E(G)$ for all distinct $x, y \in V(G)$. Recall that a \emph{path} is a nonempty graph $P=(V,E)$ of the form
$$
V=\{x_0,x_1,...,x_k\}, \quad E=\{\{x_0,x_1\},...,\{x_{k-1},x_k\}\},
$$
where $x_i$ are all distinct. The number of edges of a path is the length. Note that the length of a path can be zero.
A Hamiltonian path is a path in the graph that visits each vertex exactly once. A finite graph $C$ is a \textit{cycle} if $|V(C)|\geq 3$ and there exists an enumeration $(v_1,v_2,...,v_n)$
 of its vertices such that
\begin{equation*}
(\{v_i,v_j\}\in E(C))\Leftrightarrow (|i-j|=1\quad \mbox{or}\quad |i-j|=n-1).
\end{equation*}
For the graph $G=(V,E)$ a \emph{Hamiltonian cycle} is a cycle which is a subgraph of $G$ that visits every vertex exactly once. A connected graph without cycles is called a tree.
A tree $T$ may have a distinguished vertex called the \emph{root}; in this case $T$ is called a \emph{rooted tree}.

Generally we  follow terminology used in~\cite{BM}. A graph $G=(V,E)$ together with a function $w\colon E\rightarrow \mathbb{R}^+$, where $\mathbb{R}^+=[0,+\infty)$, is called a \textit{weighted} graph, and  $w$ is called a \textit{weight} or a \textit{weighting function}. The weighted graphs we denote by $(G,w)$.  %A weighted path we shall call a \emph{characteristic path} if the weights of all edges are different.

A nonempty graph $G$ is called \emph{complete $k$-partite} if its vertices can be divided into $k$ disjoint nonempty subsets $X_1,...,X_k$ so that there are no edges joining the vertices of the same subset $X_i$ and any two vertices from different $X_i,X_j$, $1\leqslant i,j \leqslant k$ are adjacent. In this case we write $G=G[X_1,...,X_k]$.

\section{Cycles in ultrametric spaces}\label{cr}
In the following we identify a finite ultrametric space $(X,d)$ with a complete weighted graph $(G_X,w_d)$ such that $V(G_X)=X$ and
\begin{equation}\label{eq2}
\forall \,  x,y\in X, \, x\neq y\colon \quad	 w_d(\{x,y\})=d(x,y).
\end{equation}
%for every pair of distinct points $x,y \in X$.
%In ~\cite{DP(MatSb)}
The following lemma was proved in~\cite{DP(MatSb)}.
%which yields some characterization of cycles in ultra\-met\-ric spaces.	
\begin{lem}\label{lem2.2}
Let $(X,d)$ be an ultrametric space with $\abs{X}\geqslant 3$. Then for every cycle $C\subseteq G_X$ there exist at least two distinct edges
$e_1, e_2 \in C$ such that
\begin{equation}\label{eq3}
w_d(e_1)=w_d(e_2)=\max\limits_{e\in E(C)}w_d(e).
\end{equation}
\end{lem}
%Let us consider the special class of weighted cycles $(C,w_d)$ which contain exactly two distinct edges $e_1, e_2$ of maximal weight $w_d(e_1)=w_d(e_2)$ %(i.e. satisfy condition~(\ref{eq3})) and additionally satisfy the following condition
%\begin{equation}\label{eq4}
%w_d(e_i)\neq w_d(e_j) \, \text{ for } i\neq j, \,\, i,j=2,\ldots,n,
%\end{equation}
%where $e_i$ are the other edges of $C$, $i=3,\ldots,n$ and $n=|C|$. In the following we call such cycles \emph{characteristic}.
%Let us define a class $TMC$ of weighted cycles $(C, w)$ by the next rule. A weighted cycle $(C, w)$ belongs to $TMC$ if and only if the following conditions hold.

We shall say that a weighted cycle $(C, w)$ is \emph{characteristic} if the following conditions hold.
\begin{itemize}
\item[(i)] There are exactly two distinct $e_1, e_2 \in E(C)$ such that \eqref{eq3} holds.
\item[(ii)] The restriction of $w$ on the set $E(C)\setminus \{e_1,e_2\}$ is strictly positive and injective.
\end{itemize}

\begin{rem}\label{r2.2} Let us explain the choice of a name for such a type of cycles.
%Consider the problem of continuation of the weight $w_d$ defined on the edges of characteristic cycle $(C,w_d)$ to ultrametric $\rho\colon V(C)\times V(C)\to \mathbb{R}^+$ as it was done in~\cite{DP(MatSb)}.
%According to
%Theorem 7
It was proved in~\cite{DP(MatSb)} that for every characteristic weighted cycle $(C,w)$
%$(C,w) \in TMC$
there is a unique ultrametric $d\colon V(C)\times V(C) \rightarrow \mathbb{R}^+$ such that
$$
d(x,y) = w(\{x,y\})
$$
for all  $\{x,y\} \in E(C)$.
%there exists only one continuation $\rho$. Moreover by Theorem 2 (see ~\cite[p. 1136]{DP(MatSb)}) the subdominant pseudoultrametric
%\begin{equation}\label{eq5}
%\rho_{w}(x,y):=
%\begin{cases}
%0 &\mbox{if } x=y \\\
%\inf\limits_{P\in \mathfrak{P}_{x,y}}(\max\limits_{e\in P}w(e)) &\mbox{if } x \neq y\\
%\end{cases}
%\end{equation}
%realize this continuation, i.\,e. $\rho=\rho_w$, where $\mathfrak{P}_{x,y}$ is the  set of the paths connecting vertices  $x$ and $y$ of the cycle $C$.
In other words we can uniquely reconstruct whole the ultrametric space $(X,d)$ by characteristic cycle $(C,w_d)\subseteq (G_X,w_d)$ if $\abs{V(C)}=\abs{X}$.
\end{rem}

%For the proof of Theorem~\ref{th1} we need the following definition.
%and Corollary~\ref{cor1} which was obtained in~\cite{GomoryHu(1961)}.
We need the following definition.

\begin{defn}[\cite{GomoryHu(1961)}]\label{d2}
Let $(X,d)$ be a finite ultrametric space. Define the graph $G_X^d$ as follows $V(G_X^d)=X$ and
$$
(\{u,v\}\in E(G_X^d))\Leftrightarrow(d(u,v)=\diam X).
$$
We call $G_X^d$ a \emph{diametrical graph} of the space $(X,d)$.
\end{defn}

\begin{lem}[\cite{GomoryHu(1961)}]\label{cor1}
Let $(X,d)$ be a finite ultrametric space, $|X|\geqslant 2$. If $(X,d)\in \mathfrak U$, then $G_X^d$ is a bipartite graph, $G_X^d=G_X^d[X_1,X_2]$ and $X_1 \in \mathfrak U$, $X_2 \in \mathfrak U$.
\end{lem}

We shall say that a weighted path $(P,w)$ is \emph{characteristic} if the weighting function $w\colon E(P)\rightarrow \mathbb R^+$ is injective and strictly positive.
%Let us denote by $VMP$ the class of all characteristic weighted paths.

The next theorem is the main result of this section.
\begin{thm}\label{th1}
Let $(X,d)$ be a finite ultrametric space with $\abs{X}\geqslant 3$. Then the following conditions are equivalent.
\begin{itemize}
  \item[(i)] $(X,d)\in \mfu$.
  \item[(ii)] There exists a characteristic Hamiltonian path in $G_X$.
  \item[(iii)] There exists a characteristic Hamiltonian cycle in $G_X$.
\end{itemize}
\end{thm}
\begin{proof}
\textbf{(i)$\Rightarrow$(ii)}. %Let us prove by induction on $|X|$.
We shall prove the implication (i)$\Rightarrow$(ii) by induction on $|X|$. Let $(X,d) \in \mathfrak U$. If $|X|=3$, then the existence of a characteristic Hamiltonian path is evident. Suppose the implication (i)$\Rightarrow$(ii) holds for $X$ with $|X|\leqslant n-1$. Let $|X|=n$. Let us prove that there exists a characteristic Hamiltonian path in $G_X$.  According to Lemma~\ref{cor1} we have
\begin{equation}\label{e2.6}
G_X^d=G^d_X[X_1,X_2],\quad |X_1|\leqslant n-1,\quad |X_2|\leqslant n-1
\end{equation}
and $X_1 \in \mathfrak U$, $X_2 \in \mathfrak U$. By the induction supposition there exist characteristic Hamiltonian paths $P_1\subseteq G_{X_1}$ and $P_2\subseteq G_{X_2}$. Let $V(P_1)=\{x_1,...,x_m\}$ and $V(P_2)=\{x_{m+1},...,x_n\}$, $1\leqslant m \leqslant n-1$. Since $G_X^d=G_X^d[X_1,X_2]$, we have
\begin{equation*}%\label{e22}
\diam X \notin \Sp{X_1} \ \, \text{ and } \ \, \diam X \notin \Sp{X_2}.
\end{equation*}
Moreover, the equality
\begin{equation}\label{e23}
\Sp{X_1}\cap \Sp{X_2}=\{0\}
\end{equation}
holds. Indeed, it is clear that
$$
0 \in \Sp{X_1}\cap \Sp{X_2},
$$
but if $|\Sp{X_1}\cap \Sp{X_2}| \geqslant 2$, then using the equality
\begin{equation}\label{e24}
\Sp{X} = \Sp{X_1}\cup \Sp{X_2} \cup \{\diam X\}
\end{equation}
and the Gomory--Hu inequality we obtain
$$
|\Sp{X}| \leqslant 1 + |X_1| + |X_2| - |X_1\cap X_2| < |X_1| + |X_2| = |X|
$$
contrary to $(X,d)\in \mathfrak U$. The equality $d(x_m,x_{m+1})=\diam X$, \eqref{e23} and \eqref{e24} imply that the path $P$ with $V(P) = \{x_1,..,x_m,x_{m+1},...,x_n\}$ is a characteristic Hamiltonian path in $G_X$.

\textbf{(ii)$\Rightarrow$(iii)}. Let $P$ be a characteristic Hamiltonian path in $G_X$ with $V(P)=\{x_1,...,x_n\}$. Consider the cycle $C=(x_1,...,x_n)$. It is clear that $C$ is Hamiltonian. According to Lemma~\ref{lem2.2} the equality
$$
w_d(\{x_1,x_n\})=\max\limits_{e\in E(P)}w_d(e)
$$
holds. This means that $C$ is characteristic.

\textbf{(iii)$\Rightarrow$(i)}. Let $(X,d)$ be a finite ultrametric space and let $C$ be a characteristic Hamiltonian cycle in $G_X$. Using Lemma~\ref{lem2.2} with this $C$ we easily show that $|\Sp{X}|=|X|$. Condition (i) follows.

%Then from~(\ref{eq3}) and~(\ref{eq4}) it follows that $\abs{\Sp{X}}\geqslant \abs{X}$. The last inequality and the Gomory-Hu inequality~(\ref{eq1}) give us  $\abs{\Sp{X}}= \abs{X}$. Condition (i) is established.
\end{proof}

With every finite ultrametric space $(X, d)$, we can associate (see~\cite{PD(UMB)}) a labeled rooted $m$-ary tree $T_X$ by the following rule. If $X=\{x\}$ is a one-point set, then $T_X$ is a tree consisting of one node $x$
considered strictly binary by definition. Let $|X|\geqslant 2$ and $G^d_X = G^d_X[X_1,...,X_k]$ be the diametrical graph of the space $(X, d)$. In this case the root of the tree $T_X$ is labeled by $\diam X$ and, moreover, $T_X$ has $k$ nodes $X_1,...,X_k$ of the first level with the labels

\begin{equation}\label{e2.7}
l_i=
\begin{cases}
\diam X_i, &\text{if } \, |X_i|\geqslant 2,\\
x, &\text{if } \, X_i\, \text{is a one-point set}\\
&\text{with the single element } \, x,
\end{cases}
\end{equation}
$i = 1,...,k$. The nodes of the first level indicated by labels $x \in X$ are leaves, and those indicated by labels $\diam X_i$ are internal nodes of the tree $T_X$. If the first level has no internal nodes, then the tree $T_X$ is constructed. Otherwise, by repeating the above-described procedure with $X_i \subset X$ corresponding to internal nodes of the first level, we obtain the nodes of the second level, etc. Since $|X|$ is finite, and the cardinal numbers $|Y|$, $Y \subseteq X$, decrease strictly at the motion along any path starting from the root, consequently all vertices on some level will be leaves, and the construction of $T_X$ is completed. The above-constructed labeled tree $T_X$ is called the \emph{representing tree} of the space $(X, d)$. We note that every element $x \in X$ is ascribed to some leaf, and all internal nodes are labeled as $r \in \Sp{X}$. In this case, different leaves correspond to different $x \in X$, but different internal nodes can have coinciding labels.

Recall that a rooted tree is \emph{strictly binary} if every internal node has exactly two children.
%The construction of representing tree can be found in~\cite{GomoryHu(1961)} and~\cite{PD(UMB)}.
Note that the correspondence between trees and ultrametric spaces is well known~\cite{GNS00,GurVyal(2012),H04}.

Define by $L_T$ the set of leaves of the tree $T$ and by $l(v)$ the label of the vertex $v$.

The proof of the following two lemmas is immediate.

\begin{lem}\label{lem1}
Let $X$ be a finite ultrametric space having a strictly binary tree $T_X$. If $v_0$ and $v_1$ are interval nodes of $T_X$ and $v_1$ is a direct successor of $v_0$ then the inequality $l(v_1)< l(v_0)$ holds.
\end{lem}
\begin{lem}\label{lem2}
Let $(X,d)$ be a finite ultrametric space with $|X|\geqslant 3$ and let $G_X^d=G_X^d[X_1, \ldots, X_k]$ be the diametrical graph of $(X,d)$. Then a tree $T_X$ is strictly binary if and only if $k=2$ and $T_{X_1}$ and $T_{X_2}$ are strictly binary.
\end{lem}

\begin{prop}\label{p1}
Let $(X,d)$ be a finite ultrametric space with $|X|\geqslant 3$. The following conditions are equivalent.
\begin{itemize}
\item [(i)] $T_X$ is strictly binary.
\item [(ii)] If $X_1\subseteq X$ and $|X_1|\geqslant 3$, then there exists a Hamiltonian cycle $C\subseteq G_{X_1}$
with exactly two edges of maximal weight.
%such that $C \in TMC$.
\item [(iii)] There is no equilateral triangle in $(X,d)$.
\end{itemize}
\end{prop}
\begin{proof}
%\textbf{(i)$\Leftrightarrow$(ii)}.
%It can be proved that $(X,d) \in \mathfrak U$ if and only if $T_X$ is strictly binary and moreover that $X_1 \in \mathfrak U$ for every subspace $X_1$ of $X$ if $X_1 \in \mathfrak U$. (see, Theorem 3.1 and Corollary 2.2 in \cite{PD(UMB)}). These results and Theorem~\ref{th1} imply (i)$\Leftrightarrow$(ii).

\textbf{(i)$\Rightarrow$(ii)}. Suppose $T_X$ is strictly binary. Let $X_1$ be a subset of $X$, $|X_1|\geqslant 3$. According to construction of $T_X$ all elements of $X_1$ are labels of leaves of $T_X$. Let $v_0$ be a smallest common predecessor for the leaves of $T_X$ labeled by elements of $X_1$. Let $v_0^1$ and $v_0^2$ be the two offsprings of $v_0$ (direct successors) and let $T_1$ and $T_2$ be the subtrees of the tree $T_X$ with the roots $v_0^1$ and $v_0^2$. Let $L_1=L_{T_1}\cap X_1$ and $L_2=L_{T_2}\cap X_1$ and let $P_1=\{x_1,...,x_m\}$ and  $P_2=\{x_{m+1},...,x_{|X_1|}\}$, $1\leqslant m \leqslant |X_1|-1$, be Hamiltonian paths in the spaces $(L_1,d)$ and $(L_2,d)$. By the property of representing trees of ultrametric spaces we have $d(x,y)=l(v_0)$ for all $x\in L_1$ and $y\in L_2$. Since $X_1 = L_1\cup L_2$, we obtain that the Hamiltonian cycle $C=(x_1,...,x_m,x_{m+1},...,x_{|X_1|})$ has exactly the two edges $\{x_1,x_{|X_1|}\}$ and $\{x_m,x_{m+1}\}$ of maximal weight.

\textbf{(ii)$\Rightarrow$(iii)}. This implication is evident.

\textbf{(iii)$\Rightarrow$(i)}. We will prove (i) by induction on $|X|$. The statement (i) evidently follows from (iii) if $|X|=3$. Assume that (iii)$\Rightarrow$(i) is satisfied for all finite ultrametric spaces $(X, d)$ with $3 \leqslant |X|  \leqslant n$, $n \in \mathbb N$. Let $G_X^d = G_X^d[X_1, \ldots, X_k]$ be the diametrical graph of $(X,d)$. Statement (i) holds if $k=2$. Indeed, since the inequality $|X_i|<|X|$ holds, the induction assumption implies that for every $i = 1, \ldots, k$, $T_{X_i}$ is a strictly binary tree. Hence if $k=2$, then $T_X$ is a strictly binary tree by Lemma~\ref{lem2}. To complete the proof it suffices to note that if $k \geqslant 3$ and $x_i \in X_i$ for $i =1,2,3$, then the points $x_1$, $x_2$, $x_3$ form an equilateral triangle with $d(x_1,x_2)=d(x_2,x_3) = d(x_3,x_1)=\diam X$.
\end{proof}

\section{The number of non-isometric $X \in \mathfrak U$ with given $\Sp{X}$.}\label{numofwe}
Let $n \in \mathbb N$ and $\mfu_{n}$ denote the class of ultrametric spaces $X\in \mfu$ such that $\abs{X}=n$.  In the present section we study the following question: how many non-isometric spaces having the same spectrum are in the class $\mfu_{n}$? Let us denote this number by $\kappa(\mfu_n)$.

%\begin{rem}
%In the article~\cite{DP(ActaMH)} the notion of \emph{weak similarity} was introduced.

\begin{defn}[\! \cite{DP(ActaMH)}]\label{def1.1}
Let  $(X,d_X)$, $(Y,d_Y)$ be metric spaces.  A bijective mapping $\Phi\colon X\to Y$ is a \emph{weak similarity} if there is a strictly increasing bijective function $f\colon \Sp{Y}\to \Sp{X}$ such that the equality
\begin{equation}\label{eq1.4}
d_X(x,y)=f(d_Y(\Phi(x),\Phi(y)))
\end{equation}
holds for all $x,y\in X$. Write $X\simeq Y$ if a weak similarity $\Phi:X \to Y$ exists.
\end{defn}

It is clear that $\simeq$ is an equivalence relation. It was proved in~\cite{DP(ActaMH)} that if $X$ and $Y$ are compact ultrametric spaces with the same spectrum, then every week similarity $\Phi\colon X \to Y$ is an isometry.
So, the main question of this section can be reformulated as follows. How many spaces are there in $\mfu_{n}$ \emph{up to weak similarity}?
%\end{rem}

\begin{prop}\label{p2}
Let $\mathfrak U_n:=\{X \in \mathfrak U: |X|=n\}$, $n \in \mathbb N$, let $\mathfrak U_n/\simeq$ be the quotient set of $\mathfrak U_n$ by $\simeq$ and let
$$
\kappa({\mathfrak{U}_n}):=\mathop{\mathrm{card}}(\mathfrak U_n/\simeq).
$$
Then the equality
\begin{equation}\label{e3.2}
\kappa({\mathfrak{U}_n}) = \sum_{k=2}^{n-1} C_{n-3}^{k-2} \kappa({\mathfrak{U}_k})\kappa({\mathfrak{U}_{n-k}})
\end{equation}
holds for every integer $n \geqslant 3$ with $\kappa({\mathfrak{U}_1}) = \kappa({\mathfrak{U}_2})=1$ and
$$
C_{n-3}^{k-2} = \frac{(n-3)!}{(k-2)!(n-k-1)!}.
$$
\end{prop}
\begin{proof}
Directly we can find the initial values
$$
\kappa({\mathfrak{U}_1}) = \kappa({\mathfrak{U}_2})=1.
$$
Let $n\geq 3$. The number $\kappa({\mathfrak{U}_n})$ coincides with the number of non-isometric $(X,d)\in \mathfrak{U}_n$ having the spectrum $\{0,1,...,n-1\}$. For every such $(X,d) \in \mathfrak{U}_n$ we write $G_X^d[X_1,X_2]$ for the diametrical graph of $(X,d)$. The inequality $n\geqslant 3$ implies that $\diam X= n-1>1$. Since
$$
\Sp{X}=\{n-1\}\cup \Sp{X_1} \cup \Sp{X_2}
$$
and
$$
\Sp{X_1}\cap \Sp{X_2} = \{0\},
$$
we may assume without loss of generality that
$$
1\in \Sp{X_1}\, \text{ and } \, 1\notin \Sp{X_2}.
$$
Let $|X_1|=k$. It follows from $1\in \Sp{X_1}$ that $k\geqslant 2$. Moreover the statement $X_2\neq \varnothing$ implies that $k\leqslant n-1$. As was noted in the second section of the paper we have
$$
X_1 \in \mathfrak{U}_k \, \text{ and } \, X_2 \in \mathfrak{U}_{n-k}.
$$
Let $\Sp{X_1}=\{0,1,n_1,...,n_{k-2}\}$ where $1<n_1<...<n_k$ (if $k\geqslant 3$). The set $\{n_1,...,n_{k-2}\}$ can be selected from the set $\{2,...,n-2\}$ in $C_{n-3}^{k-2}$ ways. It is clear that if $(X,d)$, $(Y,\rho)\in \mathfrak{U}_n$ and
$$
\Sp{X}=\Sp{Y}=\{0,1,...,n-1\}
$$
and if for the diametrical graphs $G_X^d[X_1,X_2]$, $G_Y^{\rho}[Y_1,Y_2]$ we have
$$
1\in \Sp{X_1} \, \text{ and } \, 1\in \Sp{Y_1},
$$
the $X$ and $Y$ are isometric if and only if $X_1$ is isometric to $Y_1$ and $X_2$ is isometric to $X_2$. Now using the multiplication principle and additional principle we obtain~(\ref{e3.2}).
\end{proof}

\begin{cor}
The number of all non-isometric spaces $X\in \mathfrak{U}_n$ with given $\Sp{X}$ equals to
$$
\sum_{k=2}^{n-1} C_{n-3}^{k-2} \kappa({\mathfrak{U}_k})\kappa({\mathfrak{U}_{n-k}}),
$$
where $\kappa({\mathfrak{U}_1}) = \kappa({\mathfrak{U}_2})=1$.
\end{cor}

Using formula \eqref{e3.2}  we can find $\kappa({\mathfrak{U}_3})=1$, $\kappa({\mathfrak{U}_4})=2$, $\kappa({\mathfrak{U}_5})=5$, $\kappa({\mathfrak{U}_6})=16$, $\kappa({\mathfrak{U}_7})=61$ and so on.

\begin{rem}
As was shown in \cite{PD(UMB)} there is an isomorphism between spaces from $\mathfrak U$ and strictly decreasing binary trees.

It is easy to see that there is also a bijection between the strictly decreasing binary trees and the ranked trees  $\mathcal R_n$. The definition of the ranked trees $\mathcal R_n$ one can find in~\cite{DW}. It was noted in~\cite{DW} that numbers of $\mathcal R_n$ correspond to sequence A000111 from~\cite{Sl}.
\end{rem}

\section{Ball-preserving mappings, $\varepsilon$-isometries and semimetric spaces}

Let $X$ be a set. A \emph{semimetric} on $X$ is a function $d\colon X\times X\to \mathbb{R}^+$ such that $d(x,y)=d(y,x)$ and $(d(x,y)=0)\Leftrightarrow (x=y)$ for all $x,y \in X$. A pair $(X,d)$, where  $d$  is a semimetric on $X$, is called a \emph{semimetric space} (see, for example, \cite{Blum(1953)}).

A \emph{directed graph} or \emph{digraph} is a set of nodes connected by edges, where the edges have a direction associated with them. In formal terms a digraph is a pair $G=(V,A)$ of
\begin{itemize}
\item  a set $V$, whose element  are called vertices or nodes,
\item a set $A$ of ordered pairs of vertices, called arcs, directed edges, or arrows.
\end{itemize}

An arc $e = \langle x, y \rangle$ is considered to be directed from $x$ to $y$; %$y$ is called the \emph{head} and $x$ is called the \emph{tail} of the arc;
$y$ is said to be a \emph{direct successor} of $x$, and $x$ is said to be a \emph{direct predecessor} of $y$. If a path made up of one or more successive arcs leads from $x$ to $y$, then $y$ is said to be a \emph{successor} of $x$, and $x$ is said to be a \emph{predecessor} of $y$.

A Hasse diagram for a partially ordered set $(X,\leqslant_X)$ is a digraph $(X,A_X)$, where $X$ is the set of vertices and $A_X\subseteq X\times X$ is  the set of directed edges such that the pair $\langle v_1,v_2 \rangle$ belongs to $A_X$ if and only if $v_1\leqslant_X v_2$, $v_1\neq v_2$, and implication
$$
(v_1\leqslant_X w\leqslant_X v_2)\Rightarrow (v_1=w \vee v_2=w)
$$
holds for every $w\in X$.

Recall that a subset $B$ of a semimetric space $(X,d)$ is called a closed ball if it can be represented as follows:
$$
B=B_r(t)=\{x\in X\colon d(x,t)\leqslant r\},
$$
where $t\in X$ and $r\in [0,\infty)$. Denote by $\textbf{B}_X$ the set of all distinct balls of
semimetric space $(X,d)$.

\begin{defn}\label{def1}
Let $X$ and $Y$ be semimetric spaces. %If for some $F:X\to Y$ and for every  $Z\in \textbf{B}_X$ we have
A mapping $F\colon X \to Y$ is ball-preserving if
\begin{equation}\label{eq1}
F(Z)\in \textbf{B}_Y,
\end{equation}
for every $Z \in \textbf{B}_X$.
%then we shall say that the mapping $F$ is ball preserving.
\end{defn}

\begin{defn}\label{def4.1}
Let $G_1 = (V_1,A_1)$ and $G_2=(V_2,A_2)$ be directed graphs. A map $F\colon V_1 \to V_2$ is a \emph{graph homomorphism} if the implication
$$
(\langle u, v\rangle \in A_1) \Rightarrow (\langle F(u), F(v)\rangle \in A_2)
$$
holds for all $u,v \in V_1$. A homomorphism $F\colon V_1 \to V_2$ is an \emph{isomorphism} if $F$ bijective and the inverse map $F^{-1}$ is also a homomorphism.
\end{defn}

According to ~\cite{HN} we shall say that a graph homomorphism $F\colon V_1 \to V_2$ from $G_1 = (V_1,A_1)$ to $G_2=(V_2,A_2)$ is \textit{arc-surjective} if for every $\langle x, y\rangle \in A_2$ there is $\langle u, v\rangle \in A_1$ such that $\langle x, y\rangle = \langle F(u), F(v)\rangle$.

It is evident that every isomorphism is arc-surjective. \emph{Furthermore, if $G_1$ and $G_2$ have no isolated points, then every injective arc-surjective homomorphism $F\colon V_1 \to V_2$ is an isomorphism.} It was shown in \cite{P(TIAMM)} that if $X$ and $Y$ are finite ultrametric spaces, then the following conditions equivalent.
\begin{itemize}
\item There is a bijective ball-preserving mapping $F\colon X \to Y$ such that the inverse mapping $F^{-1}\colon Y \to X$ is also ball-preserving.
\item The Hasse diagrams $(\B_X,A_{\B_X})$ and $(\B_Y,A_{\B_Y})$ of the posets $(\B_X,\subseteq)$ and $(\B_Y,\subseteq)$ are isomorphic as directed graphs.
\end{itemize}

\begin{defn}\label{def4.2}
Let $(X,d)$ and $(Y, \rho)$ be semimetric spaces and let $\varepsilon>0$. A surjective mapping $F\colon X\to Y$ is an $\varepsilon$-isometry if the inequality
$$
|d(x,y) - \rho(F(x),F(y))| \leqslant \varepsilon
$$
holds for all $x, y \in X$.
\end{defn}

The main result of the present section is the following two theorems.

\begin{thm}\label{th4.2}
Let $X$ be a finite nonempty semimetric space. Then there is a finite ultrametric space $Y$ and a surjective ball-preserving function $F\colon Y\to X$ such that the mapping
$$
\B_Y \ni B \mapsto F(B) \in \B_X
$$
is an arc-surjective homomorphism from the Hasse diagram $(\B_Y,A_{Y})$ of $(\B_Y, \subseteq)$ to the Hasse diagram $(\B_X,A_{X})$ of $(\B_X, \subseteq)$.
\end{thm}

\begin{thm}\label{th4.4}
Let $(Y,d)$ be a finite ultrametric space. Then for every $\varepsilon>0$ there is a bijective $\varepsilon$-isometry $\Phi\colon W \to Y$ such that $W \in \mathfrak U$.
\end{thm}

Theorems \ref{th4.2} and \ref{th4.4} imply the following
\begin{cor}\label{cor4.5}
For every finite nonempty semimetric space $X$ and every $\varepsilon>0$ there are mappings $F\colon Y\to X$ and $\Phi\colon Z \to Y$ such that $Y$ is finite and ultrametric, $Z \in \mathfrak U$, $F$ is ball-preserving, $\Phi$ is an $\varepsilon$-isometry and
$$
X = F(\Phi(Z)).
$$
\end{cor}

The next lemma will be used in the proof of Theorem~\ref{th4.2}.

\begin{lem}\label{lem4.1}
Let $X$ be a finite semimetric space. If $B \in \B_X$ and $|B|\geqslant 2$, then the following statements hold.
\begin{itemize}
\item[(i)] The ball $B$ has at least two direct predecessors in the Hasse diagram $(\B_X, A_{\B_X})$.
\item[(ii)] The union of all direct predecessors of $B$ coincides with $B$.
\end{itemize}
\end{lem}
\begin{proof}
Let $B \in \mathbf{B}_X$ and $|B|\geqslant 2$. The set of all direct predecessors of $B$ is simply the set of all maximal elements of the subset
\begin{equation}\label{e4.1}
\mathbf{S} = \{S \in \B_X\colon S \subseteq B \text{ and } S \neq B\}
\end{equation}
of the poset $(\B_X, \subseteq)$. The inequality $|B|\geqslant 2$ implies
$$
B \subseteq \bigcup S, \quad S \in \mathbf{S},
$$
because $\{x\} \in \mathbf S$ for every $x \in B$. Since $X$ is finite, $\mathbf{S}$ is also finite and consequently for every $x \in B$ there is a maximal element $S$ of $\mathbf{S}$ such that $x \in S$. Statement (ii) follows. Now to finish the proof it suffices to note that if $B$ contains a unique direct predecessor $S$, then $B = S$ contrary to \eqref{e4.1}.
\end{proof}

\begin{proof}[Proof of Theorem \ref{th4.2}]
Let $(\B_X,A_{\B_X})$ be a Hasse diagram of the poset $(\B_X,\subseteq)$.  To this diagram we assign $n$-ary rooted labeled tree $T$ by the following procedure.
Let the root $v_0$ of $T$ be labeled by $X$. Let $B_1,...,B_k$ be direct predecessors of $X$ in $(\B_X,A_{\B_X})$. Define $v_1,..,v_k$ to be the children (nodes of the first level) of $v_0$ with the labels $B_1,...,B_k$ respectively. Let us look at the nodes of the first level of the tree $T$. Define the children of the nodes $v_i$, $i=1,...,k$,  as follows: if there is no $Y$ such that $\langle Y, B_i\rangle \in A_{\B_X}$ then $v_i$ is a leaf of $T$; if $B_{i1}, B_{i2},...,B_{in}$ are direct predecessors of $B_i$ in $(\B_X,A_{\B_X})$, then define $v_{i1}, v_{i2},...,v_{in}$ to be the children of $v_i$ (nodes of the second level) with labels $B_{i1}, B_{i2},...,B_{in}$ respectively. Note that the nodes of the second level may have the identical labels in the case when $B_{ij}$ is a direct predecessor both $B_{k_1}$ and $B_{k_2}$. Do the same procedure with the nodes of the second level and so on. By
%proposition~\ref{prop1}
Lemma~\ref{lem4.1} $T$ is $n$-ary tree with $n\geqslant 2$. Note also that the leaves of $T$ are labeled with the balls $\{x_i\}$, $x_i\in X$.

Let $n$ be the number of leaves of $T$. We define a new names $y_i$, $i=1,..,n$, for the leaves of $T$ in any order but save the labels of these leaves. Let $Y$ be an ultrametric space with representing tree isomorphic to $T$, $Y=\{y_1,...,y_n\}$. Define $F\colon Y\to X$ by the rule
$$
F(y_i)=x_i\, \text{ if the label of } \,  y_i \, \text{ is } \, x_i.
$$
We claim that $F$ is ball-preserving. Indeed, by Lemma 4 in~\cite{P(TIAMM)} for every $B\in \mathbf{B}_Y$ there exists a node $\tilde{v}$ of $T$ such that $\Gamma_{T}(\tilde{v})=B$, where $\Gamma_T(\tilde{v})$ is the set of all leaves of subtree with the root $\tilde{v}$. And let $\tilde{B}$ be the label of $\tilde{v}$. According to Lemma~\ref{lem4.1} and the construction of $T$ the set $F(B)$ coincides with  $\tilde{B}$. It suffice to note that $\tilde{B}$ is a ball in $\textbf{B}_X$ because all the nodes in $T$ are labeled by balls of semimetric space $X$. Furthermore, it is easily seen that the mapping
$$
\mathbf{B}_X \ni B \mapsto F(B) \in \mathbf{B}_Y
$$
is an arc-surjective homomorphism from $(\mathbf{B}_Y, A_Y)$ to $(\mathbf{B}_X, A_X)$ as required.
\end{proof}

\begin{defn}\label{def4.8}
Let $(Y, d_Y)$ and $(W, d_W)$ be bounded metric spaces and let $\Delta>0$. The Gromov--Hausdorff distance $d_{GH}(Y,W)$ is less than $\Delta$ if there exists a metric spaces $(Z, d_Z)$ with subspaces $Y'$ and $W'$ such that
\begin{itemize}
\item $Y$ and $Y'$ are isometric;
\item $W$ and $W'$ are isometric;
\item We have the inclusions
\begin{equation}\label{e4.3}
Y' \subseteq \bigcup_{w \in W'} O_\Delta (w) \text{ and } W' \subseteq \bigcup_{y \in Y'} O_\Delta(y)
\end{equation}
where for $t \in Z$, $O_\Delta(t) = \{z \in Z\colon d_Z(t,z)<\Delta\}$ is an open ball from $(Z, d_Z)$ that has the radius $\Delta$.
\end{itemize}
\end{defn}

The next lemma is a reformulation Proposition~4.1 from \cite{PD(UMB)}.

\begin{lem}\label{lem4.9}
Let $Y$ be a finite ultrametric space and let $\varepsilon>0$. Then there is a finite ultrametric space $W \in \mathfrak U$ such that $|Y| = |W|$ and
$$
d_{GH}(Y, W) < \varepsilon.
$$
\end{lem}
Now we are ready to prove Theorem~\ref{th4.4}.

\begin{proof}[Proof of Theorem~\ref{th4.4}]
The theorem is trivial if $|Y| \leqslant 2$. Let $|Y| \geqslant 3$, let $\varepsilon>0$ and let
$$
\delta = \min\{d_Y(x,y)\colon x, y \in Y, x\neq y\}.
$$
Since $3 \leqslant |Y| < \infty$, we have $0<\delta<\infty$. By Lemma~\ref{lem4.9} for every $\Delta$ from the interval $(0, \min(\frac{\delta}{2},\frac{\varepsilon}{2}))$ there is $W \in \mathfrak U$ such that $d_{GH}(Y, W)< \Delta$. Let $(Z, d_Z)$ be metric space with contains isometric copies $Y'$ and $W'$ of $Y$ and $W$ respectively such that inclusions \eqref{e4.3} hold. We claim that for every $w \in W'$ there is a unique $y \in Y'$ such that $y \in O_\Delta(w)$. Suppose we can find $w \in W$ and two distinct $y_1, y_2 \in Y'$ which satisfy
$$
y_1 \in O_\Delta(w) \text{ and } y_2 \in O_\Delta(w).
$$
Then the triangle inequality and the definitions of $\delta$ and $\Delta$ imply
$$
\delta \leqslant d_Z(y_1,y_2) \leqslant d_Z(y_1,w) + d_Z(w,y_2) \leqslant 2 \Delta < \delta.
$$
This contradiction shows that, for every $w \in W'$, the set
$$
O_\Delta(w) \cap Y'
$$
is either empty or contains a single point. Consequently, if there exists $w^* \in W'$ such that
$$
O_\Delta(w^*) \cap Y' =\varnothing,
$$
then from the first inclusion in \eqref{e4.3} it follows that
$$
|Y'| = \left|\bigcup_{w \in W'} O_\Delta(w)\cap Y'\right| = \sum_{\substack{w \in W' \\ w\neq w^*}} |O_\Delta(w) \cap Y'| \leqslant |W'|-1,
$$
contrary to $|Y'|=|Y|=|W|=|W'|$.

Let $\varphi\colon W\to W'$ and $\psi\colon Y \to Y'$ be isometries. We define a function $\Phi\colon W \to Y$ by setting
\begin{equation}\label{e4.4}
(\Phi(w)=y) \Leftrightarrow (\psi(y) \in O_\Delta(\varphi(w)))
\end{equation}
for all $w \in W$ and $y \in Y$. The first part of the proof shows that this definition is correct and $\Phi$ is bijective. It remains to prove that $\Phi$ is an $\varepsilon$-isometry. For this purpose note that if $w_1, w_2 \in W$ and $y_1 = \Phi(w_1)$, $y_2=\Phi(w_2)$ then
$$
d_W (w_1,w_2) = d_Z(\varphi(w_1), \varphi(w_2)),
$$
$$
d_Y (\Phi(w_1), \Phi(w_2)) = d_Z(\psi(\Phi(w_1)), \psi(\Phi(w_2)))
$$
and, by \eqref{e4.4},
$$
d_Z(\varphi(w_i), \psi(\Phi(w_i))) < \Delta
$$
for $i=1,2$.  Now using the triangle inequality and the inequality $\Delta < \frac{\varepsilon}{2}$ we obtain
\begin{multline*}
|d_W(w_1,w_2) - d_Y(\Phi(w_1), \Phi(w_2))| \\
= |d_Z(\varphi(w_1), \varphi(w_2)) - d_Z(\psi(\Phi(w_1)), \psi(\Phi(w_2)))| \\
\leqslant d_Z(\varphi(w_1), \psi(\Phi(w_1))) + d_Z(\varphi(w_2), \psi(\Phi(w_2))) < \varepsilon.
\end{multline*}
Thus $\Phi$ is an $\varepsilon$-isometry as required.
\end{proof}

The class $\mathfrak U$ consisting of finite ultrametric spaces which are extremal for the Gomory-Hu inequality can be extended by the following way. If $X$ is a compact ultrametric space, then we define $X \in \mathfrak U_C$ if $Y \in \mathfrak U$ for every finite $Y \subseteq X$. It was shown in \cite{PD(UMB)} that $Y \in \mathfrak U$ if $Y \subset X$ and $X \in \mathfrak U$. Hence the class $\mathfrak U$ is a subclass of $\mathfrak U_C$. The following conjecture seems to be a natural generalization of theorems \ref{th4.2} and \ref{th4.4}.
\begin{conjecture}\label{con4.1}
Let $X$ be a compact nonempty semimetric space and let $\varepsilon>0$. Then there are continuous mappings $F\colon Y \to X$ and $\Phi\colon W \to Y$ such that $Y$ is compact ultrametric, $W \in \mathfrak U_C$, $\Phi$ is an $\varepsilon$-isometry and $F$ is ball-preserving and
$$
\mathbf{B}_Y \ni B \mapsto F(B) \in \mathbf{B}_X
$$
is an arc-surjective homomorphism from $(\mathbf{B}_Y, A_Y)$ to $(\mathbf{B}_X, A_X)$.
\end{conjecture}

This statement can be considered as a variation of the following ``universal'' property of the Cantor set: ``Any compact metric space is a continuous image of the Cantor set.''

\section*{Acknowledgements}
The research of the first author was supported by a grant received from TUBITAK within 2221-Fellowship Programme for Visiting Scientists and Scientists on Sabbatical Leave. The research of the second author was supported as a part of EUMLS project with grant agreement PIRSES -- GA -- 2011 -- 295164.

%-----------------------------------------------------------
%\bibliographystyle{unsrt}
%\bibliography{xbib}

\begin{thebibliography}{10}

\bibitem{GomoryHu(1961)}
R.{\,}E. Gomory and T.{\,}C. Hu.
\newblock Multi-terminal network flows.
\newblock {\em SIAM}, 9(4):551--570, 1961.

\bibitem{PD(UMB)}
E.~Petrov and A.~Dovgoshey.
\newblock {On the Gomory-Hu inequality}.
\newblock {\em J. Math. Sci.}, 198(4):392--411, 2014;
translation from {\em Ukr. Mat. Visn.} 10(4):469--496, 2013.

\bibitem{BM}
J.{\,}A. Bondy and U.{\,}S.{\,}R. Murty.
\newblock {\em Graph theory}, volume 244 of {\em Graduate Texts in
  Mathematics}.
\newblock Springer, New York, 2008.

\bibitem{DP(MatSb)}
А.~Dovgoshey and E.~Petrov.
\newblock Subdominant pseudoultrametric on graphs.
\newblock {\em Sb. Math.}, 204(8):1131--1151, 2013;
translation from {\em Mat. Sb.}, 204(8):51--72, 2013.

\bibitem{GNS00}
R.{\,}I. Grigorchuk, V.{\,}V. Nekrashevich, and V.{\,}I. Sushanskyi.
\newblock Automata, dynamical systems, and groups.
\newblock {\em Proc. Steklov Inst. Math.}, 231(4):128--203, 2000;
translation from {\em Tr. Mat. Inst. Steklova}, 231:134--214, 2000.

\bibitem{GurVyal(2012)}
V.~Gurvich and M.~Vyalyi.
\newblock Characterizing (quasi-)ultrametric finite spaces in terms of
  (directed) graphs.
\newblock {\em Discrete Appl. Math.}, 160(12):1742--1756, 2012.

\bibitem{H04}
Bruce Hughes.
\newblock Trees and ultrametric spaces: a categorical equivalence.
\newblock {\em Adv. Math.}, 189(1):148--191, 2004.

\bibitem{DP(ActaMH)}
Oleksiy Dovgoshey and Evgeniy Petrov.
\newblock Weak similarities of metric and semimetric spaces.
\newblock {\em Acta Math. Hungar.}, 141(4):301--319, 2013.

\bibitem{DW}
F.~Disanto and Thomas Wiehe.
\newblock Exact enumeration of cherries and pitchforks in ranked trees under
  the coalescent model.
\newblock {\em Mathematical Biosciences}, 242:195--200, 2013.


\bibitem{Sl}
N.~J.~A. Sloane.
\newblock The on-line encyclopedia of integer sequences.
\newblock {\em Notices Amer. Math. Soc.}, 50(8), 2003.
\newblock ISSN 0002-9920.

\bibitem{Blum(1953)}
Leonard~M. Blumenthal.
\newblock {\em Theory and applications of distance geometry}.
\newblock Oxford, Clarendon Press, 1953.

\bibitem{HN}
P.~Hell and J.~Ne\v{s}et\v{r}il.
\newblock {\em Graphs and Homomorphisms}.
\newblock Oxford Lecture Series in Mathematics and its Applications, Oxford Univercity Press, 2004.

\bibitem{P(TIAMM)}
E.{\,}A. Petrov.
\newblock Ball-preserving mappings of finite ulrametric spaces.
\newblock {\em Proceedings of IAMM}, 26:150--158, 2013.
\newblock (In Russian).
\end{thebibliography}

\end{document}